%% file: Polytopes2.tex
\documentclass{amsart}

\usepackage{amsmath,amssymb,amscd}

\usepackage{hyperref}

\usepackage{comment}
\usepackage{graphicx,color}
\usepackage[all]{xy}
\usepackage{tikz}
\usepackage{tikz-cd}

\usepackage[initials]{amsrefs}

\newtheorem{theorem}{Theorem}[section]
\newtheorem{lemma}[theorem]{Lemma}
\newtheorem{proposition}[theorem]{Proposition}
\newtheorem{cor}[theorem]{Corollary}

\theoremstyle{definition}
\newtheorem{definition}[theorem]{Definition}
\newtheorem{remark}[theorem]{Remark}

\newtheorem{setup}[theorem]{Setup}

\newcommand\cE{\mathcal{E}}

\newcommand\cI{\mathcal{I}}
\newcommand\cL{\mathcal{L}}

\newcommand\cO{\mathcal{O}}
\newcommand\cT{\mathcal{T}}

\renewcommand\AA{\mathbb{A}}
\newcommand\CC{\mathbb{C}}
\newcommand\GG{\mathbb{G}}

\newcommand\PP{\mathbb{P}}

\newcommand\RR{\mathbb{R}}
\newcommand\ZZ{\mathbb{Z}}

\newcommand\rG{\mathrm{G}}
\newcommand\rH{\mathrm{H}}
\newcommand\rL{\mathrm{L}}

\newcommand\rV{\mathrm{V}}

\newcommand\rmm{\mathrm{m}}

\newcommand{\epsi}{\varepsilon}
\newcommand{\inv}{^{-1}}
\newcommand{\into}{\hookrightarrow} 
\newcommand{\onto}{\twoheadrightarrow} 
\newcommand{\Ext}{\mathrm{Ext}}
\newcommand{\cExt}{\mathcal{E}xt}
\newcommand{\cHom}{\mathcal{H}om}
\DeclareMathOperator{\Hom}{Hom}  
\DeclareMathOperator{\SSpec}{Spec}
\newcommand{\cSym}{\mathcal{S}ym}
\newcommand{\Pic}{\mathrm{Pic}}
\newcommand{\Def}{Def}
\newcommand{\art}{(\mathrm{Art})}

\DeclareMathOperator{\Spec}{Spec} 
\DeclareMathOperator{\coker}{coker}

\def\conv#1{\mathrm{conv} \left\{ #1  \right\}} 
\def\spanR#1{\mathrm{span}_\mathbb{R} \left( #1 \right)} 
\def\cone#1{\mathrm{cone} \left\{ #1 \right\}}

\title{Some examples of non-smoothable Gorenstein Fano toric threefolds}

\author{Andrea Petracci}

\address{Freie Universit\"at Berlin, Institut f\"ur Mathematik, Arnimallee 3, 14195 Berlin, Germany}

\email{andrea.petracci@fu-berlin.de}

\begin{document}

\begin{abstract}
We present a combinatorial criterion on reflexive polytopes of dimension 3 which gives a local-to-global obstruction for the smoothability of the corresponding Fano toric threefolds. As a result, we show an example of a singular Gorenstein Fano toric threefold which has compound Du Val, hence smoothable, singularities but is not smoothable.
\end{abstract}

\maketitle

\section{Introduction}

In this note we consider a specific feature of the deformation theory of Fano toric threefolds with Gorenstein singularities. Such varieties are in one-to-one correspondence with the 4319 reflexive polytopes of dimension 3, which were classified by Kreuzer and Skarke \cite{kreuzer_skarke_reflexive_3topes}. 

Fix such a polytope $P$ and denote by $X_P$ the corresponding Fano toric variety, i.e.\ the toric variety associated to the spanning fan of $P$. The singularities of $X_P$ are detected by the shape of the facets of $P$. Here we will ignore the problem of understanding which singularities are smoothable. Instead, we will present a local-to-global obstruction to the smoothability of $X_P$. In other words, we will show examples where there exists an open non-affine subscheme $Y \into X_P$ such that $Y$ is singular, $Y$ has smoothable singularities, and $Y$ is not smoothable (and consequently $X_P$ is not smoothable). These examples are constructed by means of the following combinatorial criterion --- the relevant definitions are given in \S\ref{sec:toric_An_bundles}.

\begin{theorem} \label{thm:main}
Let $P$ be a reflexive polytope of dimension $3$ and let $X_P$ be the Fano toric threefold associated to the spanning fan of $P$. If, for some integer $n \geq 1$, the polytope $P$ has ``two adjacent almost-flat $A_n$-triangles'' as facets, then $X_P$ is not smoothable.
\end{theorem}

A particular polytope, which satisfies the hypothesis of Theorem~\ref{thm:main}, allows us to prove the following result.

\begin{theorem}\label{thm:example}
There exists a singular Fano toric threefold $X$ such that the singular locus of $X$ is isomorphic to $\PP^1$, $X$ has only $cA_1$-singularities, and every infinitesimal deformation of $X$ is trivial.
In particular, $X$ is not smoothable.
\end{theorem}

This refutes a conjecture made by Prokhorov \cite[Conjecture~1.9]{prokhorov_degree_Fano_threefolds}, according to which all Fano threefolds with only compound Du Val singularities are smoothable. This conjecture was motivated by Namikawa's result \cite{namikawa_fano} on the smoothability of Fano threefolds with Gorenstein terminal singularities.

\subsection*{Idea of the proof of Theorem~\ref{thm:main}} Fix an integer $n \geq 1$. An $A_n$-triangle (see Definition~\ref{def:A_n_triangle}) corresponds, via toric geometry, to the $cA_n$ threefold singularity
$
\Spec \CC[x,y,z,w]/(xy-z^{n+1}).
$

If a reflexive polytope $P$ of dimension $3$ has two adjacent $A_n$-triangles as facets, then there is an open non-affine toric subscheme $Y$ of $X_P$ such that the singular locus of $Y$ is isomorphic to $\PP^1$ and the singularities are transverse $A_n$. Here $A_n$ denotes the affine toric surface $
\Spec \CC[x,y,z]/(xy-z^{n+1})$.
 More precisely, $Y$ is an $A_n$-bundle over $\PP^1$ (see Definition~\ref{def:A_n_bundle}), i.e.\ there exists a map $\pi \colon Y \to \PP^1$ such that, Zariski locally on the target, it is the trivial projection with fibre $A_n$.
The map $\pi$ may be globally non-trivial, depending on the relative position of the two adjacent $A_n$-triangles. 
It is possible to express the sheaf $\pi_* \cExt^1_Y(\Omega^1_Y, \cO_Y)$, which is a vector bundle on $\PP^1$ of rank $n$, in terms of the combinatorics of the two triangles.
In particular, we get to know when this sheaf is the direct sum of negative line bundles on $\PP^1$. This gives a combinatorial condition for $\cExt^1_Y(\Omega^1_Y, \cO_Y)$ not to have global sections; the condition is expressed by insisting that the two triangles almost lie on the same plane, i.e.\ they are ``almost-flat'' (see Definition~\ref{def:two_adjacent_A_n_triangles}). If this happens, then every infinitesimal deformation of $Y$ is locally trivial and, thus, $X_P$ is not smoothable.

\subsection*{Relation to Mirror Symmetry for Fano varieties}
In the context of Mirror Symmetry for Fano varieties \cite{mirror_symmetry_and_fano_manifolds, conjectures}, Akhtar--Coates--Galkin--Kasprzyk \cite{sigma} introduced the notion of ``mutation''.
Starting from some combinatorial datum, a mutation transforms a Fano polytope (i.e.\ the lattice polytope associated to a Fano toric variety) into another Fano polytope.
Varying the combinatorial datum gives different mutations of the same Fano polytope.

In the setting of Theorem~\ref{thm:main}, if a 3-dimensional  reflexive polytope $P$ has two adjacent $A_n$-triangle facets ($n \geq 1$), then these are almost-flat  if and only if the polytope $P$ does not admit a special kind of mutation, which we will not specify here. Therefore, Theorem~\ref{thm:main} says that, in some cases, a Gorenstein Fano toric threefold is not smoothable if the corresponding polytope does not admit a special kind of mutation. This agrees with Ilten's observation \cite{ilten_sigma} that mutations of Fano polytopes induce deformations of the corresponding Fano toric varieties.

\subsection*{Higher dimensions}
The methods of this paper could be easily adapted to study obstructions to deformations of toric $A_n$-bundles on smooth toric varieties of any dimension. This would give a local-to-global obstruction to the smoothability of toric varieties of dimension $d \geq 4$ which contain, as an open toric subscheme, a toric $A_n$-bundle over a smooth toric variety of dimension $d-2$.

\subsection*{Notation and conventions}

We work over $\CC$, but everything will hold over a field of characteristic zero or over a perfect field of large characteristic.
If $N$ is a lattice, its dual is denoted by $M := \Hom_\ZZ(N, \ZZ)$ and the symbol $\langle \cdot , \cdot \rangle$ denotes the duality pairing between $M$ and $N$.

\subsection*{Acknowledgements} 
The results in this note have appeared in my Ph.D.\ thesis~\cite{petracci_phd}, which was supervised by Alessio Corti; I would like to thank him for suggesting this problem to me and for sharing his ideas. I am grateful to Victor  Przyjalkowski for bringing Prokhorov's conjecture to my attention.

The author was funded by Tom Coates' ERC Consolidator Grant~682603 and by Alexander Kasprzyk's EPSRC Fellowship~EP/N022513/1.

\section{$A_n$-bundles and their deformations} \label{sec:deformations_A_n_bundles}

For any integer $n \geq 1$, let $A_n$ denote the toric surface singularity associated to the cone
spanned by $(0,1)$ and $(n+1,1)$
inside the lattice $\ZZ^2$, i.e.\ the affine hypersurface
\[
A_n = \Spec \CC [ x,y,z] / (xy-z^{n+1}).
\]
The conormal sequence of the closed embedding $A_n \into \AA^3$ produces a free resolution of $\Omega_{A_n}^1$:
\begin{equation} \label{eq:conormal_sequence_A_n}
0 \longrightarrow I/I^2 = \cO_{A_n} \overset{\begin{pmatrix} y \\ x \\ -(n+1)z^n \end{pmatrix}}{\xrightarrow{\hspace*{2.5cm}}} \left. \Omega^1_{\AA^3} \right\vert_{A_n} = \cO_{A_n}^{\oplus 3} \longrightarrow \Omega^1_{A_n} \longrightarrow 0
\end{equation}
where $I$ is the ideal of $A_n$ in $\AA^3$.
This allows us to compute
\begin{align*}
\Ext^1_{A_n}(\Omega^1_{A_n}, \cO_{A_n}) = \coker \left( \cO_{A_n}^{\oplus 3} \overset{(y,x,-(n+1)z^n)}
{\xrightarrow{\hspace*{19mm}}} \cO_{A_n} \right) 
= \cO_{A_n} / (y,x, z^n) 
= \cO_{D_n}
\end{align*}
where $D_n \simeq \Spec \CC[z]/(z^n)$ is the closed subscheme of $A_n$ defined by the ideal generated by $y$, $x$ and $z^n$. Notice that $D_n$ is the singular locus of $A_n$ equipped with the schematic structure given by the second Fitting ideal of $\Omega^1_{A_n}$.

We want to define the notion of an $A_n$-bundle and globalise this computation of the Ext group.
Informally, an $A_n$-bundle is a morphism $Y \to S$ which, Zariski-locally, is the projection $A_n \times S \to S$. More precisely we have to insist that an $A_n$-bundle is a closed subscheme in a split vector bundle over $S$ of rank 3.

\begin{definition} \label{def:A_n_bundle}
An \emph{$A_n$-bundle} over a $\CC$-scheme $S$ is a morphism of schemes $\pi_Y \colon Y \to S$ such that there exist three line bundles $\cL_x, \cL_y, \cL_z \in \Pic(S)$, a closed embedding of $S$-schemes
\[
\iota \colon Y \into E = \SSpec_{S} \cSym^\bullet_{\cO_S} (\cL_x \oplus \cL_y \oplus \cL_z)^\vee
\]
of $Y$ into the total space of $\cL_x \oplus \cL_y \oplus \cL_z$, and an 
affine open cover $\{ S_i \}_i$ of $S$ satisfying the following condition: 
for each $i$, 
there are trivializations $\cL_x \vert_{S_i} \simeq \cO_{S_i}$, $\cL_y \vert_{S_i} \simeq \cO_{S_i}$, $\cL_z \vert_{S_i} \simeq \cO_{S_i}$ and a commutative
 diagram of $S_i$-schemes
 \begin{equation*}
\begin{tikzcd}
\pi_Y\inv (S_i) 
\arrow[r, "\simeq"]
\arrow[d, hook, "\iota_{S_i}"'] 
&
 \Spec \cO_{S_i}(S_i)[x_i,y_i,z_i] / (x_i y_i - z_i^{n+1}) 
 \arrow[d, hook] 
 \\
\pi_E\inv (S_i)
 \arrow[r, "\simeq"]
  &
   \Spec \cO_{S_i}(S_i) [x_i, y_i, z_i] = \AA^3_{S_i}
\end{tikzcd}
\end{equation*}
where $\pi_E$ denotes the projection $E \to S$, the coordinates $x_i \in \Gamma(S_i, \cL_x^\vee)$, $y_i \in \Gamma(S_i, \cL_y^\vee)$ and $z_i \in \Gamma(S_i, \cL_z^\vee)$ are the local sections corresponding to the trivializations above, the horizontal arrows are isomorphisms, the left vertical arrow is the restriction of the closed embedding $\iota \colon Y \into E$, and the right vertical arrow is the base change of the standard embedding $A_n \into \AA^3$ to $S_i$.
\end{definition}

\begin{remark} \label{rmk:cocycle_condition_A_n_bundles}
A posteriori one can see that $\cL_x \otimes \cL_y \simeq \cL_z^{\otimes (n+1)}$. This follows from the following easy fact in commutative algebra: let $A$ be a ring and $f \in A$ be an invertible element; if the ideal of $A[x,y,z]$ generated by $xy-z^{n+1}$ coincides with the ideal generated by $xy - f z^{n+1}$, then $f = 1$.
\end{remark}

\begin{lemma} \label{lemma:thickening_zero_section_line_bundle}
Let $S$ be a scheme with a line bundle $\cL \in \Pic(S)$. Let $D$ be the $k$-th  order thickening of the zero section of the total space of $\cL$, i.e.\ the closed subscheme of $\SSpec_S \cSym^{\bullet}_{\cO_S} \cL^\vee$ locally defined by the equation $x^{k+1} = 0$ where $x$ is a nowhere vanishing local section of $\cL^\vee$. Let $\pi \colon D \to S$ be the projection. Then
\begin{equation*}
\pi_* \cO_D = \bigoplus_{i=0}^k (\cL^\vee)^{\otimes i}.
\end{equation*}
\end{lemma}

\begin{proof}
Let $\{ S_i \}_i$ be an affine open cover of $S$ which trivializes $\cL$. Let $x_i \in \Gamma(S_i, \cL^\vee)$ be a local coordinate. Then we have the isomorphism of $S_i$-schemes
\[
\pi\inv(S_i) \simeq \Spec \cO_S(S_i)[x_i]/(x_i^{k+1}).
\]
Therefore $\pi_* \cO_D \vert_{S_i}$ is the free $\cO_{S_i}$-module with basis $\{1, x_i, \dots, x_i^k \}$, which is a local frame of $\cO_S \oplus \cL^\vee \oplus \cdots \oplus (\cL^\vee)^{\otimes k}$.

Another way to see this is to notice that $D = \SSpec_S (\cSym^\bullet_{\cO_S} \cL^\vee) / \cI$, and consequently $\pi_* \cO_D = (\cSym^\bullet_{\cO_S} \cL^\vee ) / \cI$, where $\cI \subseteq \cSym^\bullet_{\cO_S} \cL^\vee$ is the ideal made up of elements of degree greater than $k$.
\end{proof}

\begin{proposition} \label{prop:pushforward_ext_sheaf_A_n_bundle}
Let $S$ be a $\CC$-scheme and $\pi_Y \colon Y \to S$ be an $A_n$-bundle, with $\cL_x, \cL_y, \cL_z \in \Pic(S)$ as in Definition~\ref{def:A_n_bundle}. Then there is an isomorphism of $\cO_S$-modules
\begin{equation*}
(\pi_Y)_* \left( \cExt^1_Y(\Omega^1_{Y/S}, \cO_Y) \right) \simeq \bigoplus_{2 \leq j \leq n+1} \cL_z^{\otimes j}.
\end{equation*}
\end{proposition}

\begin{proof}
Assume we are in the setting of Definition~\ref{def:A_n_bundle}, with projections $\pi_Y \colon Y \to S$ and $\pi_E \colon E \to S$, closed embedding $\iota \colon Y \into E$, and a trivialising affine open cover $\{ S_i \}_i$ of $S$ with local sections $x_i, y_i, z_i$.

We consider the conormal sequence of $Y \overset{\iota}{\into} E \overset{\pi_E}\rightarrow S$:
\begin{equation} \label{eq:conormal_sequence_S}
\cI_{Y/E} / \cI_{Y/E}^2 \longrightarrow \Omega^1_{E/S} \vert_Y \longrightarrow \Omega^1_{Y/S} \longrightarrow 0,
\end{equation}
where $\cI_{Y/E}$ is the ideal sheaf of the closed embedding $\iota \colon Y \into E$.
We restrict this sequence to $S_i$ and we get the conormal sequence of $Y_i = \pi_Y\inv (S_i) \overset{\iota_{S_i}}{\into} E_i = \pi_E\inv(S_i) \rightarrow S_i$:
\begin{equation} \label{eq:conormal_sequence_S_i}
\cI_{Y_i/E_i} / \cI_{Y_i/E_i}^2 \longrightarrow \Omega^1_{E_i/S_i} \vert_{Y_i} \longrightarrow \Omega^1_{Y_i/S_i} \longrightarrow 0;
\end{equation}
this is the base change to $S_i$ of \eqref{eq:conormal_sequence_A_n}, the conormal sequence of $A_n \into \AA^3 \to \Spec \CC$. As $S_i \to \Spec \CC$ is flat, we have that \eqref{eq:conormal_sequence_S_i} is left exact for all $i$. As $\{ S_i \}_i$ is an open cover of $S$, we have that also \eqref{eq:conormal_sequence_S} is left exact.

Since $\pi_E \colon E \to S$ is the vector bundle whose sheaf of sections is $\cL_x \oplus \cL_y \oplus \cL_z$, we have that $\Omega^1_{E/S} = \pi_E^* (\cL_x \oplus \cL_y \oplus \cL_z)^\vee$. Therefore $\Omega^1_{E/S} \vert_Y = \pi_Y^* (\cL_x \oplus \cL_y \oplus \cL_z)^\vee$.

One can check that $\cI_{Y/E} / \cI_{Y/E}^2 \simeq \pi_Y^* (\cL_x \otimes \cL_y)^\vee$. On the intersection $S_{ij} = S_i \cap S_j$ we have the equalities 
$ x_i = g_{ij}^x x_j$,
  $y_i = g_{ij}^y y_j$, and
   $z_i = g_{ij}^z z_j$,
where $g_{ij}^x, g_{ij}^y, g_{ij}^z \in \Gamma(S_{ij}, \cO_S^*)$ are invertible functions such that $g_{ij}^x g_{ij}^y = (g_{ij}^z)^{n+1}$ (by Remark~\ref{rmk:cocycle_condition_A_n_bundles}). Then the restriction of the map
\[
\pi_Y^* (\cL_x \otimes \cL_y)^\vee = \cI_{Y/E} / \cI_{Y/E}^2 \longrightarrow \Omega^1_{E/S} \vert_Y = \pi_Y^* (\cL_x \oplus \cL_y \oplus \cL_z)^\vee
\]
in \eqref{eq:conormal_sequence_S} to $Y_{ij} = \pi_Y\inv (S_{ij})$ produces the following commutative diagram.
\begin{equation*}
\begin{tikzcd}
\cO_{Y_{ij}}  \arrow{d}[swap]{g_{ij}^x g_{ij}^y}       \arrow{rr}{\begin{pmatrix}
y_i \\ x_i \\ -(n+1) z_i^n
\end{pmatrix}}   &    &
\cO_{Y_{ij}}^{\oplus 3}   \arrow[d, "\mathrm{diag}(g_{ij}^x{,}g_{ij}^y{,}g_{ij}^z)"]    \\
\cO_{Y_{ij}}      \arrow{rr}[swap]{\begin{pmatrix}
y_j \\ x_j \\ -(n+1) z_j^n
\end{pmatrix}}     &  &
\cO_{Y_{ij}}^{\oplus 3}     
\end{tikzcd}       
\end{equation*}
Therefore the sequence \eqref{eq:conormal_sequence_S} becomes
\begin{equation*}
0 \longrightarrow \pi_Y^* (\cL_x \otimes \cL_y)^\vee \longrightarrow \pi_Y^* (\cL_x \oplus \cL_y \oplus \cL_z)^\vee \longrightarrow \Omega^1_{Y/S} \longrightarrow 0,
\end{equation*}
which gives a locally free resolution of $\Omega^1_{Y/S}$. Hence
\begin{align*}
\cExt^1_Y(\Omega^1_{Y/S}, \cO_Y) &= \coker \left( \pi_Y^* ( \cL_x \oplus \cL_y \oplus \cL_z) \longrightarrow \pi_Y^* (\cL_x \otimes \cL_y) \right) \\
&= \pi_Y^* (\cL_x \otimes \cL_y) \otimes_{\cO_Y} \cO_D \\
&= \pi_Y^* (\cL_z)^{\otimes (n+1)} \otimes_{\cO_Y} \cO_D
\end{align*}
where $D \into Y$ is the closed subscheme locally defined by $x_i = y_i = z_i^n = 0$.
Denote by $\pi_D \colon D \to S$ the projection. It is clear that $D$ is the $(n-1)$-th order thickening of the zero section in the total space $\cL_z$ over $S$. By Lemma~\ref{lemma:thickening_zero_section_line_bundle} we have
\begin{equation*}
(\pi_D)_* \cO_D = \bigoplus_{i=0}^{n-1} (\cL_z^\vee)^{\otimes i}.
\end{equation*}

Thus
\begin{align*}
(\pi_Y)_* \cExt^1_Y (\Omega^1_{Y/S}, \cO_Y) &= (\pi_Y)_* (\pi_Y^* \cL_z^{\otimes (n+1)} \otimes_{\cO_Y} \cO_D) \\
&= (\pi_D)_* (\pi_D^* \cL_z^{\otimes (n+1)}) \\
&= (\pi_D)_* \cO_D \otimes_{\cO_S} \cL_z^{\otimes (n+1)} \\
&= \bigoplus_{i=0}^{n-1} (\cL_z^\vee)^{\otimes i} \otimes_{\cO_S} \cL_z^{\otimes (n+1)} \\
&= \bigoplus_{2 \leq j \leq n+1} \cL_z^{\otimes j}.
\end{align*}
This concludes the proof of Proposition~\ref{prop:pushforward_ext_sheaf_A_n_bundle}.
\end{proof}

The following lemma is well known in deformation theory.

\begin{lemma} \label{lemma:lci_scheme_with_no_Ext_have_locally_trivial_deformations}
Let $Y$ be a reduced $\CC$-scheme. Assume that $Y \to \Spec \CC$ is a local complete intersection morphism and that $\rH^0(Y, \cExt^1_Y(\Omega^1_Y, \cO_Y))=0$.

Then all infinitesimal deformations of $Y$ are locally trivial. In particular, if $Y$ is not smooth, then $Y$ is not smoothable.
\end{lemma}

\begin{proof}
Let $\art$ be the category of local artinian $\CC$-algebras with residue field $\CC$.
Let $\Def_Y$ be the functor of infinitesimal deformations of $Y$, i.e.\ the covariant functor from $\art$ to the category of sets which maps each $A \in \art$ to the set $\Def_Y(A)$ of isomorphism classes of deformations of $Y$ over $\Spec A$ and acts on arrows by base change.
For every $A \in \art$, let $\Def_Y'(A)$ be the subset of $\Def_Y(A)$ made up of the locally trivial deformations. This gives a subfunctor $\phi \colon \Def_Y' \into \Def_Y$. We refer the reader to \cite[\S 2.4]{sernesi_deformation} or to \cite{manetti_seattle} for details.

We want to show that the natural transformation $\phi$ is an isomorphism.
It is enough to show that the injective function $\phi_A \colon \Def_Y'(A) \into \Def_Y(A)$ is surjective for every $A \in \art$.
This is implied by the smoothness of $\phi$ (see \cite[Definition~3.9]{manetti_seattle}). This is what we will prove below.

Let $\cT_Y = \cHom_Y(\Omega^1_Y, \cO_Y)$ be the sheaf of derivations on $Y$. By \cite[Theorem 2.4.1]{sernesi_deformation} the tangent space of $\Def_Y'$ is $\rH^1(Y, \cT_Y)$ and the tangent space of $\Def_Y$ is $\Ext^1_Y(\Omega^1_Y, \cO_Y)$. 
By \cite[Proposition 2.4.6]{sernesi_deformation}, $\rH^2(Y, \cT_Y)$ is an obstruction space for $\Def_Y'$. 
By \cite[Proposition~2.4.8]{sernesi_deformation} or \cite[Theorem 4.4]{vistoli_deformation_lci}, $\Ext^2_Y(\Omega^1_Y, \cO_Y)$ is an obstruction space for $\Def_Y$.

The local-to-global spectral sequence for Ext gives the following five term exact sequence
\begin{align*}
0 &\longrightarrow \rH^1(Y, \cT_Y) \longrightarrow \Ext^1_Y(\Omega_Y^1, \cO_Y) \longrightarrow \rH^0(Y, \cExt^1_Y(\Omega_Y^1, \cO_Y)) \longrightarrow \\
&\longrightarrow \rH^2(Y, \cT_Y) \longrightarrow \Ext^2_Y(\Omega_Y^1, \cO_Y).
\end{align*}
With the identifications above, the vanishing of $\rH^0(Y, \cExt^1_Y(\Omega_Y, \cO_Y))$ implies that $\phi$ induces an isomorphism on tangent spaces and an injection on obstruction spaces. By \cite[Remark~4.12]{manetti_seattle} we get that $\phi$ is smooth.
\end{proof}

\begin{cor} \label{cor:zerocohomology_An_bundle_non_smoothability}
Let $S$ be a smooth $\CC$-scheme and $\pi_Y \colon Y \to S$ be an $A_n$-bundle, with $\cL_x, \cL_y, \cL_z \in \Pic(S)$ as in Definition~\ref{def:A_n_bundle}. Then we have:
\begin{enumerate}
\item[(i)] the sheaf $\cExt^1_Y (\Omega^1_Y, \cO_Y)$ is isomorphic to $\cExt^1_Y (\Omega^1_{Y/S}, \cO_Y)$;
\item[(ii)] if $\rH^0(S, \cL_z^{\otimes j}) = 0$ for all $2 \leq j \leq n+1$, then all infinitesimal deformations of $Y$ are locally trivial and $Y$ is not smoothable.
\end{enumerate}
\end{cor}

\begin{proof}
As $Y \to S$ is a Zariski-locally trivial fibration, the sequence of K{\"a}hler differentials of $Y \to S \to \Spec \CC$ is left exact and locally split:
\begin{equation*}
0 \longrightarrow \pi_Y^* \Omega^1_S  \longrightarrow \Omega^1_Y \longrightarrow \Omega^1_{Y/S} \longrightarrow 0.
\end{equation*}
This implies that the dual sequence
\begin{equation*}
0 \longrightarrow \cHom_Y (\Omega^1_{Y/S}, \cO_Y) \longrightarrow
\cHom_Y(\Omega^1_{Y}, \cO_Y) \longrightarrow
\cHom_Y(\pi^*_Y \Omega^1_S, \cO_Y) \longrightarrow 0
\end{equation*}
is exact. From the long exact sequence of Ext sheaves we get the following exact sequence of $\cO_Y$-modules:
\begin{equation*}
0 \longrightarrow \cExt^1_Y (\Omega^1_{Y/S}, \cO_Y) \longrightarrow \cExt^1_Y(\Omega^1_Y, \cO_Y) \longrightarrow \cExt^1_Y(\pi_Y^* \Omega^1_S, \cO_Y). 
\end{equation*}
But the last sheaf is zero because $S$ is smooth over $\CC$. This proves (i).

By Proposition~\ref{prop:pushforward_ext_sheaf_A_n_bundle} we deduce that 
\[
\rH^0 (Y,\cExt^1_Y(\Omega^1_Y, \cO_Y)) = \bigoplus_{2 \leq j \leq n+1} \rH^0 (S, \cL_z^{\otimes j}) = 0.
\]
From Lemma~\ref{lemma:lci_scheme_with_no_Ext_have_locally_trivial_deformations} we deduce (ii).
\end{proof}

\section{Toric $A_n$-bundles over $\PP^1$} \label{sec:toric_An_bundles}

\begin{definition} \label{def:A_n_triangle}
Fix an integer $n \geq 1$ and a 3-dimensional lattice $N$. An \emph{$A_n$-triangle} in $N$ is a lattice triangle $T \subseteq N_\RR$ such that:
\begin{enumerate}
\item there are no lattice points in the relative interior of $T$;
\item the edges of $T$ have lattice lengths $1$, $1$, and $n+1$;
\item $T$ is contained in a plane which has height $1$ with respect to the origin, i.e.\ there exists a linear form $w \in M = \Hom_\ZZ(N,\ZZ)$ such that $T$ is contained in the affine plane $H_{w,1} := \{ v \in N_\RR \mid \langle w, v \rangle = 1 \}$, where $\langle \cdot, \cdot \rangle$ is the duality pairing between $M$ and $N$.
\end{enumerate}
\end{definition}

\begin{figure}
\includegraphics[scale=0.8]{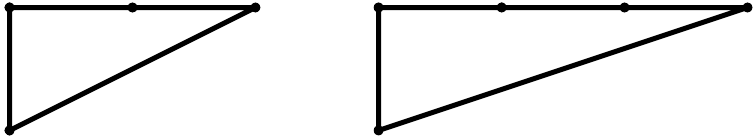}
\caption{An $A_1$-triangle and an $A_2$-triangle}
\end{figure}

If $T$ is an $A_n$-triangle in the 3-dimensional lattice $N$, consider the cone $\sigma \subseteq N_\RR$ spanned by the vertices of $T$. Then the affine toric variety associated to the cone $\sigma$, namely $\Spec \CC[\sigma^\vee \cap M]$, is isomorphic to
$
\Spec \CC[x,y,z,w] / (xy - z^{n+1})
$; every point with $x=y=z=0$ is a $c A_n$ singularity.

\begin{definition} \label{def:two_adjacent_A_n_triangles}
Fix an integer $n \geq 1$ and a 3-dimensional lattice $N$. \emph{Two adjacent $A_n$-triangles} in $N$ are two $A_n$-triangles $T_0$ and $T_1$ in $N$ such that:
\begin{enumerate}
\item[(4)] $T_0 \cap T_1$ is the edge of length $n+1$ for both $T_0$ and $T_1$;
\item[(5)] $T_0$ and $T_1$ lie in the two different half-spaces of $N_\RR$ defined by the plane $\spanR{T_0 \cap T_1}$.
\end{enumerate}
We say that $T_0$ and $T_1$ are \emph{almost-flat} if $\langle w_1, \rho_0 \rangle = 0$, where $\rho_0$ is the vertex of the triangle $T_0$ not in the segment $T_0 \cap T_1$ and $w_1 \in M$ is the linear form such that $T_1$ is contained in the plane $H_{w_1,1}$.
\end{definition}

Notice that the condition of almost-flatness is symmetric between $T_0$ and $T_1$ because $\langle w_1, \rho_0 \rangle = \langle w_0, \rho_1 \rangle$. 

\begin{remark}
Let $P$ be a reflexive polytope in the lattice $N$ of rank $3$ and let $T_0$ and $T_1$ be two adjacent $A_n$-triangles which are facets of $P$. The convexity of $P$ implies $\langle w_1, \rho_0 \rangle \leq 0$.

Consider the dual polytope
\[
P^* = \{ u \in M_\RR \mid \forall v \in P, \ \langle u, v \rangle \geq -1 \}.
\]
The dual face of $T_0$ (resp. $T_1$) is the vertex $-w_0$ (resp. $-w_1$) of $P^*$. The dual face of the edge $T_0 \cap T_1$ is the edge $\conv{-w_0, -w_1}$ of $P^*$. The segment $\conv{-w_0, -w_1}$ has lattice length equal to $1 - \langle w_1, \rho_0 \rangle$.
\end{remark}

\begin{figure}
\begin{center}
\def\svgwidth{6cm}
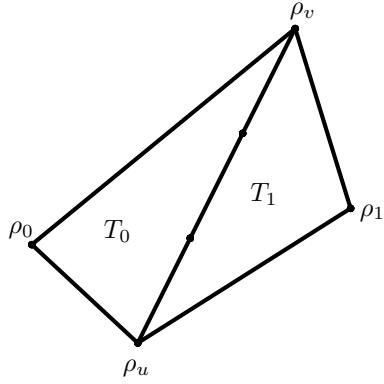
\end{center}
\caption{Two adjacent $A_2$-triangles}
\label{fig:two_adjacent_triangles}
\end{figure}

\begin{setup} \label{set:toric_An_bundle_P1}
Let $T_0$ and $T_1$ be two adjacent $A_n$-triangles in a 3-dimensional lattice $N$. We denote by $\rho_u$ and $\rho_v$ the vertices of the segment $T_0 \cap T_1$. Let $\rho_0$ (resp. $\rho_1$) be the vertex of $T_0$ (resp. $T_1$) which does not lie on $T_0 \cap T_1$ (see Figure~\ref{fig:two_adjacent_triangles}). Let $Y$ be the toric variety associated to the fan in $N$ generated by $\cone{\rho_0, \rho_u, \rho_v}$ and $\cone{\rho_1, \rho_u, \rho_v}$. The projection $N \to N / (N \cap (\RR \rho_u + \RR \rho_v)) \simeq \ZZ$ induces a toric morphism $\pi \colon Y \to \PP^1$.
\end{setup}

\begin{proposition} \label{prop:toric_A_n_bundles}
Let $T_0$ and $T_1$ be two adjacent $A_n$-triangles in a 3-dimensional lattice $N$. Then the toric morphism $\pi \colon Y \to \PP^1$, constructed in Setup~\ref{set:toric_An_bundle_P1}, is an $A_n$-bundle and there exists an isomorphism
\begin{equation} \label{eq:isom_fascio_Ext_P1}
\pi_* \cExt^1_Y(\Omega^1_Y, \cO_Y) \simeq \bigoplus_{2 \leq j \leq n+1} \cO_{\PP^1}\left( -j \left( \langle w_1, \rho_0 \rangle + 1 \right) \right).
\end{equation}
Moreover, if $\langle w_1, \rho_0 \rangle \geq 0$ then all infinitesimal deformations of $Y$ are locally trivial and $Y$ is not smoothable.
\end{proposition}

Before proving this proposition we prove the following lemma.

\begin{lemma} \label{lemma:GL3_transformation_vectors_An_bundle}
After a $\rG \rL_3(\ZZ)$-transformation, in Setup~\ref{set:toric_An_bundle_P1} we may assume that $N = \ZZ^3$ and
\begin{equation*}
\rho_{0} = \begin{pmatrix}
a \\ b \\ -1
\end{pmatrix}\!\!, \
\rho_{1} = \begin{pmatrix} 0 \\ 0 \\ 1 \end{pmatrix}\!\!, \
\rho_u = \begin{pmatrix} 1 \\ 0 \\ 0 \end{pmatrix}\!\!, \
\rho_v = \begin{pmatrix}  -n \\  n+1 \\ 0 \end{pmatrix}\!\!, 
\end{equation*}
for some $a,b \in \ZZ$.
\end{lemma}

\begin{proof}
Let $\hat{\rho} \in N$ be the lattice point on the segment between $\rho_u$ and $\rho_v$ which is the closest one to $\rho_u$. The triangle with vertices $\rho_u, \rho_1, \hat{\rho}$ is an empty triangle at height $1$, so $\{ \rho_u, \rho_{x_1}, \hat{\rho} \}$ is a basis of $N$. Without loss of generality we may assume that $\rho_u = (1,0,0)$, $\hat{\rho} = (0,1,0)$ and $\rho_1 = (0,0,1)$. Since on the edge between $\rho_u$ and $\rho_v$ there are $n+2$ lattice points, we have $\rho_v = \rho_u + (n+1) (\hat{\rho} - \rho_u) = (-n, n+1, 0)$.

Assume $\rho_0 = (a,b,c)$ for some $a,b,c \in \ZZ$. Since $\rho_u, \hat{\rho}, \rho_0$ are the vertices of an empty triangle at height $1$, they constitute a basis of $N$. Therefore $c = \det ( \rho_u \vert \hat{\rho} \vert \rho_0 ) = \pm 1$.

Since $\rho_0$ and $\rho_1$ have to be in the two different half-spaces in which the plane $\RR \rho_u + \RR \rho_v = (0,0,1)^\perp$ divides $N_\RR$, we have $c<0$, so $c=-1$.
\end{proof}

\begin{proof}[Proof of Proposition~\ref{prop:toric_A_n_bundles}]
By Lemma~\ref{lemma:GL3_transformation_vectors_An_bundle}, the ray map $\ZZ^4 \to N = \ZZ^3$ of $Y$ is given by the matrix
\begin{equation*}
\begin{pmatrix}
a & 0 & 1 & -n \\
b & 0 & 0 & n+1 \\
-1 & 1 & 0 & 0
\end{pmatrix}\!.
\end{equation*}
One can see that the ideal of $\ZZ$ generated by the $2 \times 2$ minors is $\ZZ$ itself and the ideal generated by the $3 \times 3$ minors is $r \ZZ$, where $r = \mathrm{gcd}(n+1,b)>0$. Let $p, q \in \ZZ$ be such that $b = r p$ and $n+1 = r q$. The kernel of the ray map is generated by the primitive vector $(q,q,-np-aq,-p)$. By B{\'e}zout let $s,t  \in \ZZ$ be such that $s p + t q = 1$. 
The cokernel of the transpose of the ray map is the homomorphism $\ZZ^4 \onto \ZZ \oplus \ZZ / r \ZZ$ given by the matrix
\begin{equation*}
\begin{pmatrix}
q & q & -qa-pn & -p \\
\bar{s} & \bar{s} & - \bar{s} \bar{a} + \bar{t} \bar{n} & \bar{t}
\end{pmatrix}\!,
\end{equation*}
where $\bar{\cdot}$ denotes the reduction modulo $r$.
By \cite[Theorem~4.1.3]{cls}, the divisor class group of $Y$ is isomorphic to $\ZZ \oplus \ZZ/r\ZZ$.

Let the group
\begin{equation*}
G = \left\{ \left. \left( \lambda^{q} \epsi^s, \lambda^{q} \epsi^s, \lambda^{-qa-pn} \epsi^{-sa+tn}, \lambda^{-p} \epsi^t \right) \in \GG_\rmm^4 \right\vert \lambda \in \GG_\rmm, \epsi \in \boldsymbol{\mu}_r \right\}
\end{equation*}
act linearly on the affine space $\AA^4 = \Spec \CC[x_0, x_1, u,v]$.
By \cite[\S5.1]{cls},  $Y$ is the geometric quotient of $\AA^4 \smallsetminus \rV(x_0, x_1) = \Spec \CC[x_0^\pm, x_1, u,v] \cup \Spec \CC[x_0, x_1^\pm, u,v]$ with respect to this action.
The variables $x_0, x_1, u, v$ can be identified with the Cox coordinates of $Y$ associated to the rays $\rho_0, \rho_1, \rho_u, \rho_v$, respectively.
The toric morphism $\pi \colon Y \to \PP^1$ is defined by
\[
[x_0 : x_1 : u : v] \mapsto [x_0 : x_1],
\]
where $[x_0 : x_1 : u : v]$ denotes the point of $Y$ corresponding to the $G$-orbit of the point $(x_0,x_1,u,v) \in \AA^4$.

We consider the following integers
\begin{align*}
d_x &= b - (n+1)(a+b), \\
d_y &= -b, \\
d_z &= -a - b.
\end{align*}
We consider the line bundles $\cL_x = \cO_{\PP^1}(d_x)$, $\cL_y = \cO_{\PP^1}(d_y)$, $\cL_z = \cO_{\PP^1}(d_z)$ and the sheaf $\cE = \cL_x \oplus \cL_y \oplus \cL_z$ on $\PP^1$. Let $\pi_E \colon E \to \PP^1$ be the total space of $\cE$ over $\PP^1$.
Then $E$ is the geometric quotient of $\Spec \CC[x_0,x_1,x,y,z] \smallsetminus \rV(x_0, x_1)$ with respect to the linear action of $\GG_\rmm$ with weights $(1,1,d_x, d_y, d_z)$.
The variables $x_0,x_1,x,y,z$ can be identified with the Cox coordinates of the toric variety $E$.
We denote by $[x_0 : x_1 : x : y : z]$ the point of $E$ corresponding to the $\GG_\rmm$-orbit of $(x_0,x_1,x,y,z) \in \AA^5$.

It is easy to check that the map $\iota \colon Y \to E$ given by
\[
[x_0 : x_1 : u : v] \mapsto [x_0 : x_1 : u^{n+1} : v^{n+1} : uv ]
\]
is a closed embedding, locally defined by $xy-z^{n+1} = 0$. So $\pi \colon Y \to \PP^1$ is an $A_n$-bundle and we are in the situation of Definition~\ref{def:A_n_bundle}.

The triangle $T_1$ is contained in the plane $H_{w_1, 1}$, where $w_1 = (1,1,1)$. Therefore $\langle w_1, \rho_0 \rangle = a+b-1 = -d_z -1$. By Proposition~\ref{prop:pushforward_ext_sheaf_A_n_bundle} and Corollary~\ref{cor:zerocohomology_An_bundle_non_smoothability} we have the isomorphism \eqref{eq:isom_fascio_Ext_P1}.

The inequality $\langle w_1, \rho_0 \rangle \geq 0$ implies that $\cL_z$ is a negative line bundle on $\PP^1$ and, by Corollary~\ref{cor:zerocohomology_An_bundle_non_smoothability}, that all infinitesimal deformations of $Y$ are locally trivial.
\end{proof}

\begin{proof}[Proof of Theorem~\ref{thm:main}]
It is an immediate consequence of Proposition~\ref{prop:toric_A_n_bundles}.
\end{proof}

\begin{remark}
There are 273 reflexive polytopes of dimension 3 which satisfy the condition of Theorem~\ref{thm:main}: the complete list is given in \cite[Remark~4.15]{petracci_phd}. Therefore, there are at least 273 non-smoothable Gorenstein Fano toric threefolds.
\end{remark}

\begin{proof}[Proof of Theorem~\ref{thm:example}]
In the lattice $N = \ZZ^3$ consider the reflexive polytope 
$P$ that is the convex hull of the following vectors:
\begin{equation*}
\rho_0 = \begin{pmatrix}
0 \\ 0 \\ 1
\end{pmatrix}\!\!, \
\rho_1 = \begin{pmatrix}
0 \\ 1 \\ -1
\end{pmatrix}\!\!, \
\rho_u = \begin{pmatrix}
0 \\ 1 \\ 0
\end{pmatrix}\!\!, \
\rho_v = \begin{pmatrix}
-2 \\ -1 \\ 0
\end{pmatrix}\!\!, \
\xi = \begin{pmatrix}
1 \\ 0 \\ 0
\end{pmatrix}\!\!.
\end{equation*}
Let $\Sigma$ be the spanning fan of $P$. The maximal cones of $\Sigma$ are:
\begin{align*}
\cone{\rho_0,\rho_u,\rho_v},   &\qquad \cone{\rho_1, \rho_u, \rho_v}, \\
\cone{\rho_0,\rho_u,\xi}, &\qquad \cone{\rho_1,\rho_u,\xi}, \\
\cone{\rho_0,\rho_v,\xi}, &\qquad \cone{\rho_1,\rho_v,\xi}.
\end{align*}
The singular cones of $\Sigma$ are the ones in the first row and $\cone{\rho_u, \rho_v}$. The corresponding facets of $P$ are two adjacent $A_1$-triangles. We have $w_1 = (-1,1,0)$ and $\langle w_1, \rho_0 \rangle = 0$, so the two $A_1$-triangles are almost flat.

Let $X$ be the Fano toric threefold associated to the fan $\Sigma$. The singular locus of $X$ is the curve $C$, which is the closure of the torus-orbit corresponding to $\cone{\rho_u,\rho_v}$. The curve $C$ is isomorphic to $\PP^1$ and the singularities of $X$ along $C$ are transverse $A_1$.

By Proposition~\ref{prop:toric_A_n_bundles} the sheaf $\cExt^1_X(\Omega^1_X, \cO_X)$ is the line bundle $\cO_{C}(-2)$ on $C$. Therefore $\rH^0(X, \cExt^1_X(\Omega^1_X, \cO_X) ) = 0$.

Let $j \colon U \into X$ be the inclusion of the smooth locus of $X$. Notice that the sheaf of derivations $\cT_X = \cHom_X(\Omega_X^1, \cO_X)$ is isomorphic to $j_* \Omega_U^2 \otimes \cO_X(-K_X)$, because these two sheaves are both reflexive and coincide on the open subset $U$ whose complement has codimension $2$. As $-K_X$ is ample, by Bott--Steenbrink--Danilov vanishing \cite[Theorem~9.3.1]{cls} we have $\rH^1 (X, \cT_X) = 0$. This argument comes from the proof of \cite[Theorem~5.1]{totaro_jumping}.

From the five term exact sequence for Ext, which is rewritten in the proof of Lemma~\ref{lemma:lci_scheme_with_no_Ext_have_locally_trivial_deformations}, we deduce that $\Ext^1_X(\Omega_X^1, \cO_X) = 0$. This implies that all infinitesimal deformations of $X$ are trivial. In particular, $X$ is not smoothable.
\end{proof}

\bibliography{Biblio_polytopes}


\end{document}

%% file: B1.pdf_tex
\begingroup%
  \makeatletter%
  \providecommand\color[2][]{%
    \errmessage{(Inkscape) Color is used for the text in Inkscape, but the package 'color.sty' is not loaded}%
    \renewcommand\color[2][]{}%
  }%
  \providecommand\transparent[1]{%
    \errmessage{(Inkscape) Transparency is used (non-zero) for the text in Inkscape, but the package 'transparent.sty' is not loaded}%
    \renewcommand\transparent[1]{}%
  }%
  \providecommand\rotatebox[2]{#2}%
  \ifx\svgwidth\undefined%
    \setlength{\unitlength}{609.70787048bp}%
    \ifx\svgscale\undefined%
      \relax%
    \else%
      \setlength{\unitlength}{\unitlength * \real{\svgscale}}%
    \fi%
  \else%
    \setlength{\unitlength}{\svgwidth}%
  \fi%
  \global\let\svgwidth\undefined%
  \global\let\svgscale\undefined%
  \makeatother%
  \begin{picture}(1,0.84418089)%
    \put(0,0){\includegraphics[width=\unitlength,page=1]{B1.pdf}}%
    \put(0.20752531,0.30841141){\color[rgb]{0,0,0}\makebox(0,0)[lb]{\smash{$T_0$}}}%
    \put(0.53211035,0.38955774){\color[rgb]{0,0,0}\makebox(0,0)[lb]{\smash{$T_1$}}}%
    \put(0.24960116,0.01237785){\color[rgb]{0,0,0}\makebox(0,0)[lb]{\smash{$\rho_u$}}}%
    \put(0.61926747,0.80430524){\color[rgb]{0,0,0}\makebox(0,0)[lb]{\smash{$\rho_v$}}}%
    \put(-0.00435659,0.32){\color[rgb]{0,0,0}\makebox(0,0)[lb]{\smash{$\rho_0$}}}%
    \put(0.7740465,0.35800091){\color[rgb]{0,0,0}\makebox(0,0)[lb]{\smash{$\rho_1$}}}%
  \end{picture}%
\endgroup%

%% file: Polytopes2.bbl
\begin{bibdiv}
\begin{biblist}

\bib{conjectures}{article}{
      author={Akhtar, Mohammad},
      author={Coates, Tom},
      author={Corti, Alessio},
      author={Heuberger, Liana},
      author={Kasprzyk, Alexander},
      author={Oneto, Alessandro},
      author={Petracci, Andrea},
      author={Prince, Thomas},
      author={Tveiten, Ketil},
       title={Mirror symmetry and the classification of orbifold del {P}ezzo
  surfaces},
        date={2016},
     journal={Proc. Amer. Math. Soc.},
      volume={144},
      number={2},
       pages={513\ndash 527},
}

\bib{sigma}{article}{
      author={Akhtar, Mohammad},
      author={Coates, Tom},
      author={Galkin, Sergey},
      author={Kasprzyk, Alexander~M.},
       title={Minkowski polynomials and mutations},
        date={2012},
        ISSN={1815-0659},
     journal={SIGMA Symmetry Integrability Geom. Methods Appl.},
      volume={8},
       pages={Paper 094, 17},
}

\bib{mirror_symmetry_and_fano_manifolds}{incollection}{
      author={Coates, Tom},
      author={Corti, Alessio},
      author={Galkin, Sergey},
      author={Golyshev, Vasily},
      author={Kasprzyk, Alexander},
       title={Mirror symmetry and {F}ano manifolds},
        date={2013},
   booktitle={European {C}ongress of {M}athematics},
   publisher={Eur. Math. Soc., Z\"urich},
       pages={285\ndash 300},
}

\bib{cls}{book}{
      author={Cox, David~A.},
      author={Little, John~B.},
      author={Schenck, Henry~K.},
       title={Toric varieties},
      series={Graduate Studies in Mathematics},
   publisher={American Mathematical Society, Providence, RI},
        date={2011},
      volume={124},
}

\bib{ilten_sigma}{article}{
      author={Ilten, Nathan~Owen},
       title={Mutations of {L}aurent polynomials and flat families with toric
  fibers},
        date={2012},
     journal={SIGMA Symmetry Integrability Geom. Methods Appl.},
      volume={8},
      number={047},
       pages={7},
}

\bib{kreuzer_skarke_reflexive_3topes}{article}{
      author={Kreuzer, Maximilian},
      author={Skarke, Harald},
       title={Classification of reflexive polyhedra in three dimensions},
        date={1998},
     journal={Adv. Theor. Math. Phys.},
      volume={2},
      number={4},
       pages={853\ndash 871},
}

\bib{manetti_seattle}{incollection}{
      author={Manetti, Marco},
       title={Differential graded {L}ie algebras and formal deformation
  theory},
        date={2009},
   booktitle={Algebraic geometry---{S}eattle 2005. {P}art 2},
      series={Proc. Sympos. Pure Math.},
      volume={80},
   publisher={Amer. Math. Soc., Providence, RI},
       pages={785\ndash 810},
}

\bib{namikawa_fano}{article}{
      author={Namikawa, Yoshinori},
       title={Smoothing {F}ano {$3$}-folds},
        date={1997},
     journal={J. Algebraic Geom.},
      volume={6},
      number={2},
       pages={307\ndash 324},
}

\bib{petracci_phd}{book}{
      author={Petracci, Andrea},
       title={On {M}irror {S}ymmetry for {F}ano varieties and for
  singularities},
      series={Ph.D. thesis, Imperial College London},
        date={2017},
}

\bib{prokhorov_degree_Fano_threefolds}{article}{
      author={Prokhorov, Yu.~G.},
       title={The degree of {F}ano threefolds with canonical {G}orenstein
  singularities},
        date={2005},
     journal={Mat. Sb.},
      volume={196},
      number={1},
       pages={81\ndash 122},
}

\bib{sernesi_deformation}{book}{
      author={Sernesi, Edoardo},
       title={Deformations of algebraic schemes},
      series={Grundlehren der Mathematischen Wissenschaften},
   publisher={Springer-Verlag, Berlin},
        date={2006},
      volume={334},
}

\bib{totaro_jumping}{article}{
      author={Totaro, Burt},
       title={Jumping of the nef cone for {F}ano varieties},
        date={2012},
     journal={J. Algebraic Geom.},
      volume={21},
      number={2},
       pages={375\ndash 396},
}

\bib{vistoli_deformation_lci}{article}{
      author={Vistoli, Angelo},
       title={The deformation theory of local complete intersections},
        note={arXiv preprint
  \href{https://arxiv.org/abs/alg-geom/9703008}{arXiv:9703008}},
}

\end{biblist}
\end{bibdiv}
